\documentclass[11pt]{article}
\usepackage{amssymb}
\usepackage{amsmath}
\usepackage{amsthm}
\usepackage{graphicx, url}
\usepackage{fullpage}

\newtheorem{theorem}{Theorem}[section]
\newtheorem{proposition}{Proposition}[section]

\newtheorem{lemma}{Lemma}[section]

\begin{document}

\title{Cops and robbers on planar directed graphs}

\author{ Po-Shen Loh \thanks{Department of Mathematical Sciences, Carnegie
Mellon University, Pittsburgh, PA 15213. E-mail: ploh@cmu.edu. Research
supported in part by an NSA Young Investigators Grant, NSF grants
DMS-1201380 and DMS-1455125, and by a USA-Israel BSF Grant.}
\and
Siyoung Oh \thanks{Daum Kakao Corp., South Korea.  Email: unbing@gmail.com.  Research conducted for Masters Thesis at Carnegie
Mellon University.}
}

\date{}

\maketitle

\begin{abstract}
  Aigner and Fromme initiated the systematic study of the \emph{cop
  number}\/ of a graph by proving the elegant and sharp result that in
  every connected planar graph, three cops are sufficient to win a natural
  pursuit game against a single robber.  This game, introduced by
  Nowakowski and Winkler, is commonly known as \emph{Cops and Robbers}\/ in
  the combinatorial literature.  We extend this study to directed planar
  graphs, and establish separation from the undirected setting.  We exhibit
  a geometric construction which shows that a more sophisticated robber
  strategy can indefinitely evade three cops on a particular strongly
  connected planar directed graph.
\end{abstract}

\section{Introduction}

The general study of pursuit games on graphs drew a substantial amount of
research attention over the last decade.  Its appeal stemmed from the
combination of its apparent proximity to natural applications, some
combinatorially elegant results and conjectures, and the challenge of
developing tools to analyze game-theoretic dynamics on graphs.  Indeed,
dynamic processes are typically already significantly more difficult to
analyze than properties of static graphs, and game theoretic interactions
between opposing parties drive the complexity to another level.

This paper considers the most extensively studied game in this area,
commonly known as \emph{Cops and Robbers}, introduced by Nowakowski and
Winkler \cite{NW}, and independently by Quillot \cite{Quillot}.  In its
classical setting, a graph is fixed, and fully known to two players,
\emph{the cops}\/ and \emph{the robber}.  The cops move first, placing $k$
cops on the vertices of the given graph, at any locations of choice
(multiple cops are allowed to reside on the same vertex).  The robber then
chooses a single vertex at which to start.  Players alternate turns,
starting with the cops, and on each turn, they choose a subset of their
agents to move across one edge each.  Note that the robber has only one
agent, and so decides whether or not to move to an adjacent vertex.  If the
robber ends up on the same vertex as a cop, then the cops win.

The main question is to determine, for each graph, the minimum value of $k$
(known as the \emph{cop number}\/ of the graph) for which there is a
strategy for the cops that guarantees a win within finite time.  This
game-theoretic graph invariant was introduced by Aigner and Fromme
\cite{AF} shortly after the game's appearance in the combinatorial
literature, and in that same paper, the authors proved the elegant and
sharp result that every planar graph has cop number at most three.

This basic setting is a natural prototype for a general class of pursuit
games on graphs, and it has been the subject of numerous papers, including
multiple surveys \cite{Alspach, BB, FT, Hahn} and ultimately a book by
Bonato and Nowakowski \cite{BN}.  Many variants have been studied,
including random graphs \cite{BKL, LuczakP, PW}, Cayley graphs
\cite{Frankl2}, geometric graphs \cite{BDFM}, directed graphs \cite{FKL},
and fast robbers \cite{AM}, to name just a few.

The central open conjecture in this area, due to Meyniel (communicated by
Frankl \cite{Frankl1}), is that every $n$-vertex graph has cop number at
most $O(\sqrt{n})$, which would be asymptotically tight.  The current best
bounds of $O(n / e^{\Theta(\sqrt{\log n})})$ were proven by Lu and Peng
\cite{LP}, with alternate proofs independently discovered by Frieze,
Krivelevich, and Loh \cite{FKL}, and Scott and Sudakov \cite{SS}.

The same paper of Frieze, Loh, and Krivelevich also formally started the
systematic study of the game in directed graphs (where the cops and robber
can only move along the direction of each edge), mainly in the context of
Meyniel's conjecture.  Specifically, the focus was on $n$-vertex strongly
connected digraphs, because the problem for a general digraph easily
reduces to the problems on its strongly connected components.  As usual,
digraphs turn out to be more complicated than undirected graphs, and they
obtained a weaker upper bound of $O(n \cdot \frac{(\log \log n)^2}{\log
n})$, which is still the current best bound for the cop number in directed
graphs.  The lower bound from undirected graphs carries over to the
directed case (simply replace each edge with a pair of antiparallel
directed edges), but there was no improvement.

This paper combines the directed graph inquiry with the original focus of
Aigner and Fromme on planar graphs.  Specifically, we ask to determine, for
each $n$, the maximum cop number of any $n$-vertex planar digraph.
Unfortunately, the approach of Aigner and Fromme for their upper bound (of
three cops) completely breaks down, as it relied on repeated clever
applications of the following simple and elegant observation.

\begin{lemma}
  In an undirected graph, if $P$ is a geodesic (a shortest path between a
  pair of vertices), then a single cop can guard all of the vertices of
  $P$: after a bounded number of turns, if the robber ever moves onto a
  vertex of $P$, it will be captured by that cop.
  \label{lem:AF-upper}
\end{lemma}

Geodesics are particularly useful in planar graphs because they can provide
powerful separation properties in the plane, thereby efficiently trapping
the robber in successively smaller regions.  The proof of Lemma
\ref{lem:AF-upper} employs the cop strategy of always moving towards the
vertex of $P$ which is nearest to the robber.  The geodesic's minimality
guarantees that the cop will eventually be able to reach every vertex of
$P$ at least as quickly as the robber.  However, this strategy is clearly
impossible in directed graphs, and indeed, there is no upper bound written
in the literature.  For completeness, we observe that the Lipton-Tarjan
Planar Separator Theorem gives a nontrivial upper bound, but it is still
far from constant.  The proof will appear in Section \ref{sec:upper}.

\begin{proposition}
  \label{prop:upper}
  Every $n$-vertex strongly connected planar digraph has cop number at most
  $O(\sqrt{n})$.
\end{proposition}

The main open question for planar digraphs is whether the cop number is
always bounded by a constant in each strongly connected component.  It is
then interesting to examine lower bound constructions.  Again, basic
methods do not work, as many previous results (including the tightness of
Meyniel's conjecture) relied on another simple observation.

\begin{lemma}
  Every undirected graph with minimum degree $\delta$ and no 3- or 4-cycles
  has cop number at least $\delta$.
  \label{lem:AF-lower}
\end{lemma}

This follows from the elementary robber strategy of remaining stationary
until a cop moves to an adjacent vertex, and then moving to another
adjacent vertex which has no cops adjacent to it.  This lemma obviously
extends to directed graphs with a similar strategy: every digraph with
minimum out-degree $\delta^+$ and an appropriate girth condition (e.g.,
undirected girth at least five) has cop number at least $1 + \delta^+$.  It
suffices to focus on constructions in which each unordered pair of vertices
induces at most one directed edge, because if antiparallel edge pairs
exist, the entire digraph can be replaced by one with no antiparallel pairs
by subdividing every edge with a new unique vertex, and replacing each
antiparallel edge pair by a pair of independent 2-edge directed paths.
Since every $n$-vertex planar graph has fewer than $3n$ edges, any such
construction will always have $\delta^+ \leq 2$, and so the standard robber
strategy cannot even be used to improve the lower bound by any amount at
all.

The main contribution of this paper is a geometrically-inspired
construction which introduces and employs a more sophisticated strategy for
the robber, and breaks through the lower bound for undirected graphs.

\begin{theorem}
  There is a strongly connected planar digraph which requires more than
  three cops to capture the robber.
  \label{thm:main}
\end{theorem}

The rest of this paper is organized as follows.  The construction is
described in the next section, and analyzed in the section thereafter.  We
close with a short proof of Proposition \ref{prop:upper} in the final
section.

\section{Construction}

We will create a geometric construction which is clearly embeddable on the
surface of a sphere, at which point a standard stereographic projection
produces a planar oriented graph.  Although it is impossible to construct a
planar oriented graph with all out-degrees at least 3, it is natural to
start with an object which is as close as possible.  Indeed, consider an
icosahedron (Figure \ref{fig:icosahedron}), which as an undirected graph is
a triangulation with all degrees equal to 5.  We use this as the starting
point for a series of steps which ultimately produce our construction.

\begin{figure}[htbp]\centering
\begin{minipage}[c]{.3\textwidth}\centering
  \includegraphics[width=.95\textwidth]{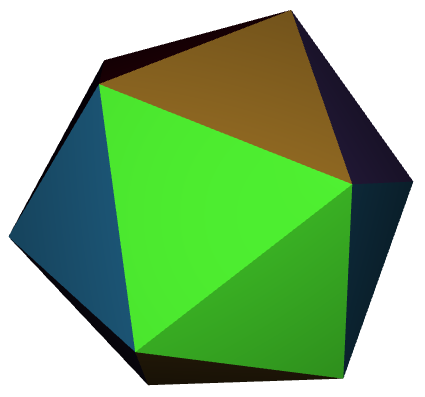}
  \caption{Icosahedron}\label{fig:icosahedron}
\end{minipage}\hspace{.1\textwidth}
\begin{minipage}[c]{.3\textwidth}\centering
  \includegraphics[width=.95\textwidth]{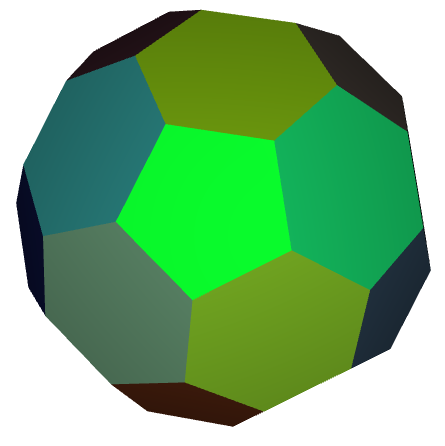}
  \caption{Truncated icosahedron}\label{fig:trunc-icosahedron}
\end{minipage}
\end{figure}

Truncate each of its vertices to obtain the truncated icosahedron in Figure
\ref{fig:trunc-icosahedron}.  Observe that each vertex is replaced with a
pentagon, and each of the triangular faces is replaced by a hexagon.  Next,
truncate again, this time along all edges of the original icosahedron
(which are now precisely the edges between pentagons).  This operation
produces the Archimedean solid in Figure \ref{fig:rhombi}, which is known
as the \emph{great rhombicosidodecahedron}, or \emph{truncated
icosidodecahedron}.\footnote{All polyhedron images were generated by
  \url{http://gratrix.net/polyhedra/webgl/poly.xhtml}.}
  
Observe that the original icosahedron vertices have been replaced by
decagons, the original icosahedron edges have been replaced by
quadrilaterals, and the original icosahedron faces have been replaced by
hexagons.  Also observe that each quadrilateral links two decagons, and
naturally identifies a pair of parallel edges between decagons, as
highlighted in Figure \ref{fig:rhombi-highlighted}.

\begin{figure}[htbp]\centering
\begin{minipage}[c]{.3\textwidth}\centering
  \includegraphics[width=.95\textwidth]{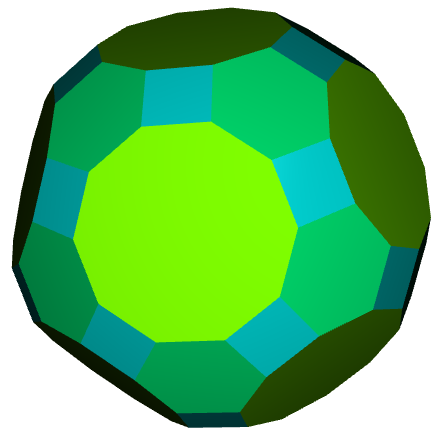}
  \caption{Great rhombicosidodecahedron}\label{fig:rhombi}
\end{minipage}\hspace{.1\textwidth}
\begin{minipage}[c]{.3\textwidth}\centering
  \includegraphics[width=.95\textwidth]{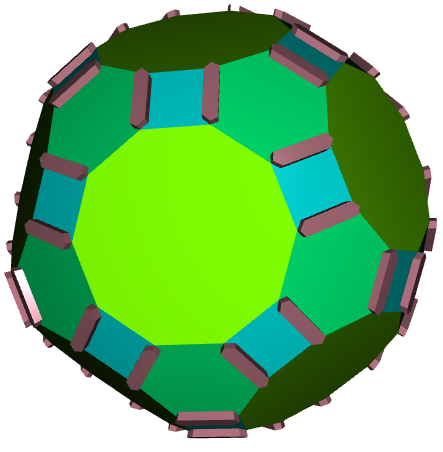}
  \caption{Highlighting parallel edges}\label{fig:rhombi-highlighted}
\end{minipage}
\end{figure}

We use this fundamental structure as the backbone for our construction.
Introduce a new vertex at the center of each decagon, called a
\emph{center}, and join it to each vertex of its decagon with a
\emph{spoke}, as in Figure \ref{fig:unit-undir}.  We will use the term
\emph{unit}\/ to refer to the whole structure of a single decagon,
including its center and spokes.  Two units are \emph{neighbors}\/ if they
are joined by one of the highlighted edges in Figure
\ref{fig:rhombi-highlighted}.

\begin{figure}[htbp]\centering
  \includegraphics[width=.6\textwidth]{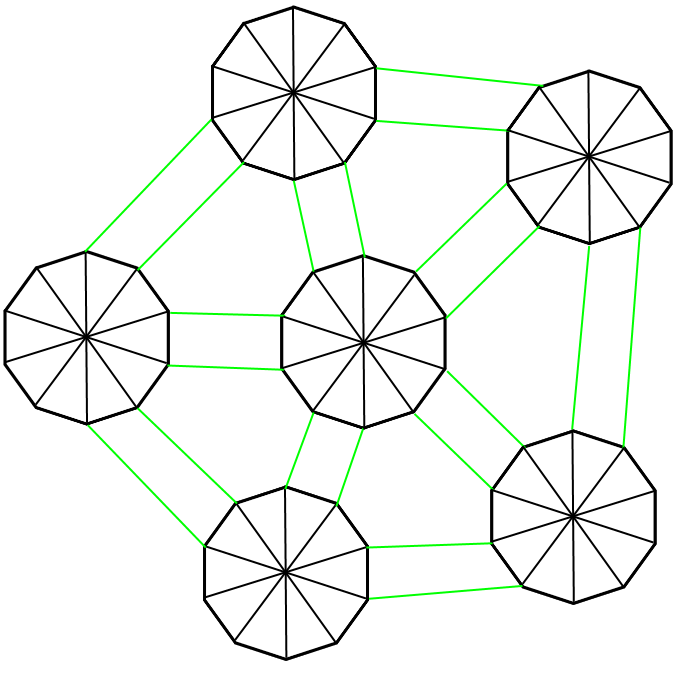}
  \caption{Point of view from any unit (without directions on edges). Note
  that edges connecting units are green.}
  \label{fig:unit-undir}
\end{figure}

The last step is to give directions and lengths to all edges.  Observe that
between each pair of neighboring units, we can always find a hexagon as
highlighted in Figure \ref{fig:unit-undir-hexagon}.  In each such hexagon,
orient all edges counter-clockwise.  This orientation is consistent because
the sphere is an orientable surface.  Ultimately, we obtain the structure
of Figure \ref{fig:unit-dir}.

\begin{figure}[htbp]\centering
  \includegraphics[width=.6\textwidth]{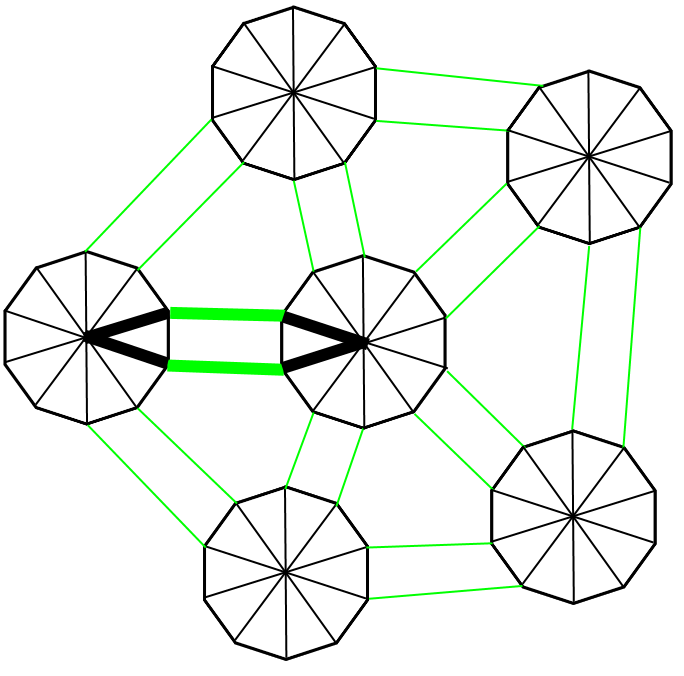}
  \caption{Point of view from any unit, highlighting one hexagon.}
  \label{fig:unit-undir-hexagon}
\end{figure}

\begin{figure}[htbp]\centering
  \includegraphics[width=.6\textwidth]{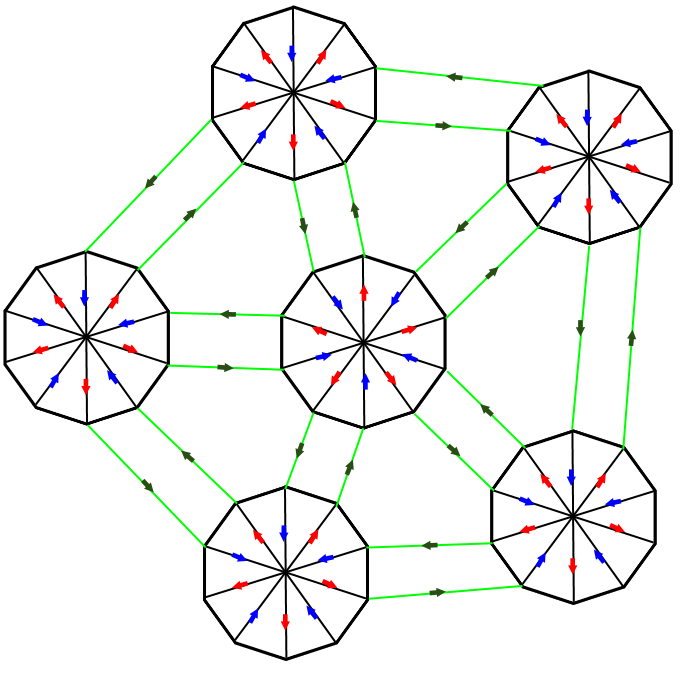}
  \caption{Point of view from any unit, with directions determined.}
  \label{fig:unit-dir}
\end{figure}

All green edges between units are now consistently oriented.  Subdivide
each of them with 999 new vertices so that it takes 1,000 turns to move
from one end of a green edge to the other end.  Likewise, subdivide each
spoke between a center and its decagon with 9 new vertices so that each
spoke now has length 10.  The only edges remaining to be oriented are the
decagon edges.  Replace each of them with a directed chain-like structure
as in Figure \ref{fig:16chain}.  It now takes 16 turns to move from one
original decagon vertex to another, and at most one additional turn to
reverse direction when traversing the chain.

\begin{figure}[htbp]\centering
  \includegraphics[width=.6\textwidth]{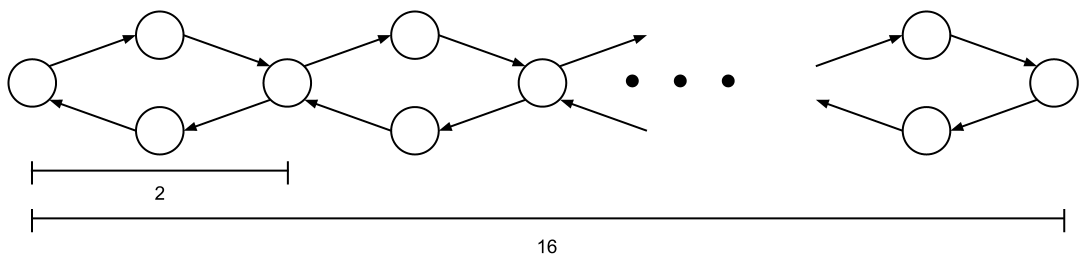}
  \caption{Replacement for one original decagon edge.}\label{fig:16chain}
\end{figure}

\section{Analysis}

In this section, we prove that our construction requires more than three
cops to capture the robber.  We achieve this by analyzing the robber's
travel from unit to unit.  Note that each unit has five neighboring units,
and in order to travel from a unit $U_1$ to a neighboring unit $U_2$, there
is exactly one \emph{exit}\/ vertex on the perimeter of $U_1$ from which
the robber can directly travel along a directed path of length 1000 to
reach $U_2$.  Each unit therefore has five exits.  It is convenient to make
the following observation.

\begin{lemma}
  Suppose that the robber is at the center of a unit $U$, and it is the
  robber's turn.  Let $c$ be the number of cops in unit $U$ which are not
  on spokes that are oriented towards the center.  Then, for any set $S$ of
  more than $c$ exits of $U$, there is at least one vertex in $S$ that the
  robber can reach in 10 turns, without being captured by any of the cops
  currently in $U$.
  \label{lem:robber-from-center}
\end{lemma}

\begin{proof}
  The robber will move directly toward one of the vertices of $S$.  It is
  clear that with this strategy, the robber cannot be stopped by any cop
  who is currently on a spoke oriented towards the center.  Furthermore,
  since each decagon edge was replaced by a chain which takes 16 turns to
  traverse from end to end, it is clear that every other cop in $U$ can
  only be within 10 moves of at most one vertex of $S$.  Since $|S| > c$,
  there will be a choice for the robber which avoids all of the cops.
\end{proof}

We will ultimately break into cases based upon how many cops are in the
robber's current unit.  It turns out that the most delicate case is when
there is exactly one cop in the robber's unit, and the following
observation will cleanly handle that situation, in conjunction with the
previous result.

\begin{lemma}
  Suppose that the robber is on the perimeter of a unit, and that unit has
  exactly one cop, located one vertex away from the center along a spoke
  which is oriented away from the center.  Then, the robber can reach the
  center within 27 moves, without that cop catching it.
  \label{lem:robber-from-perim}
\end{lemma}

\begin{proof}
  The robber will follow a shortest directed path from its current location
  to the center.  Since the spokes alternate in orientation toward and away
  from the center, and it only takes at most one turn to change direction
  along a chain on the perimeter, it is clear that the robber will reach
  the center without any interference from the cop, who starts 9 moves away
  from the perimeter along an outward spoke.  The maximum number of moves
  required for the robber is 27, because it may take the robber 1 move to
  change direction along a perimeter chain, 16 moves to traverse the
  perimeter to the nearest inward spoke, and 10 more moves to traverse the
  inward spoke to the center.
\end{proof}

We are now ready to prove that three cops are insufficient to capture the
robber on our construction.

\begin{proof}[Proof of Theorem \ref{thm:main}.]
  At the start of the game, three cops select their positions.  Since the
  icosahedron has 12 vertices, the robber is free to select a position
  which is the center of a unit that starts with no cops.  The analysis now
  proceeds by considering how the robber moves from unit to neighboring
  unit.

  The robber remains at a center until a cop arrives at an adjacent vertex
  (along an inbound spoke in its unit).  It is then the robber's turn.  It
  suffices to show that the robber can always move from this state to a
  state in which it is again at a center (possibly of a different unit),
  with a cop adjacent along an inbound spoke, and it is the robber's turn.
  This will prove that the robber can escape capture indefinitely.

  So, let us focus on the situation in which the robber is at a center of
  some unit $U$, with an immediately adjacent inbound cop.  If there are
  any cops on length-1000 unit-to-unit directed paths towards $U$, without
  loss of generality, assume that they are already at the corresponding
  entry vertices along the perimeter of $U$.  (This only makes it more
  difficult for the robber to escape to a neighboring unit center, because
  cops that are trapped in transit along a unit-to-unit path have nowhere
  to go but forward to $U$.)  Note that there are five exit vertices from
  $U$, each to a distinct neighboring unit.  The robber's strategy will be
  to use one of them to exit to a neighboring unit which cannot be reached
  more quickly by any cop.
  
  We split into three cases, depending on how many cops are in unit $U$.
  If all three cops are in $U$, then Lemma \ref{lem:robber-from-center}
  implies that of the five exits from $U$, there is at least one option
  which can be reached with no interference from the cops.  The robber
  moves directly toward one of these, takes the length-1000 directed path
  to the neighboring unit, and then moves directly toward that unit's
  center, completing this case.

  Next, consider the case when there are exactly two cops in $U$.  The cop
  outside $U$ is in some unit $U'$, but even if one considers $U'$ together
  with its five neigboring units, these six total units overlap with at
  most three neighboring units of $U$.  (We have used the fact that in an
  icosahedron, among the five neighbors of a fixed vertex $u$, the overlap
  size with a different vertex $u'$ and its neighborhood is largest when
  $u'$ is a neighbor of $u$, at which point it has size three.)  The robber
  will seek an exit which does not go to $U'$ or a neighboring unit of
  $U'$.  Since there were five exits, there are still at least two left.
  Lemma \ref{lem:robber-from-center} then implies that there is at least
  one option which can be reached with no interference from the cops in
  $U$, and the robber safely proceeds through that exit to the center of
  that neighboring unit.

  The final case has exactly one cop in $U$.  The previous argument no
  longer works, because each of the two cops outside $U$ can in theory
  block up to three neighboring units of $U$ (as in the previous case's
  analysis), and could together block all five neighboring units of $U$.
  The robber counters with a different strategy.  It is this twist in this
  case which improves the lower bound from three to four, and here we
  leverage the length-1000 paths between neighboring units.  The key
  insight is that as long as no other cops start moving toward $U$, the
  robber can evade the single cop in $U$ indefinitely.  However, the moment
  a cop starts down a long one-way street towards $U$, it stops guarding
  three neighboring units, and in fact guards zero neighboring units.  We
  formalize this as follows.

  The robber begins by moving directly along an outbound spoke, and reaches
  a vertex on the perimeter of unit $U$ in 10 moves.  It then stays
  stationary until the cop in $U$ moves onto an adjacent vertex.  It then
  moves along the perimeter of $U$, in a direction away from the cop,
  moving only when the cop moves onto an adjacent vertex along an edge
  directed towards it.  (So, it is possible that the robber spends a
  substantial amount of time not moving at all, if the cop is moving
  through vertices which are not adjacent to the robber.)

  The robber continues this simple evasion strategy until one of the other
  two cops takes a step into a length-1000 unit-to-unit directed path
  leading to $U$.  At this point, the robber switches strategy again, to
  exit $U$ within 200 turns.  To achieve this, observe that the remaining
  cop who is neither in $U$ nor trapped in the length-1000 path to $U$ can
  reach at most three neighboring units of $U$ within 1500 turns.  Let $S$
  be the set of exits of $U$ which do not lead to those units.
  
  The robber selects a direction away from the cop in $U$, and consistently
  moves along the perimeter in that direction until he reaches one of the
  (at least two) exits in $S$.  If during this process, the cop never
  passes through the center, then the robber will definitely reach its exit
  without being captured by the cop in $U$.  Since each decagon edge was
  replaced by a 16-turn chain, this will take at most 160 turns.

  Otherwise, if the cop attempts to route through the center, the robber
  suddenly changes strategy again at the moment the cop moves onto a vertex
  adjacent to the center along an outbound spoke (which must happen on any
  route through the center).  At that moment, the robber employs the
  strategy in Lemma \ref{lem:robber-from-perim}, and reaches the center
  within 27 more moves without interference from that cop.  Then, the
  robber changes strategy again to that in Lemma
  \ref{lem:robber-from-center}, and since $|S| > 1$, it will definitely be
  able to reach an exit in $S$ in 10 more turns, without any interference
  from the cop in $U$.  Therefore, the robber will be able to reach an exit
  in $S$ within a total of 200 turns, during which the cop en route from a
  neighboring unit (along a length-1000 directed path) is still far off.
  The robber then traverses the length-1000 directed path out of this exit,
  and reaches the center of the corresponding neighboring unit within 1500
  turns, without any interference from any cops, as claimed.
\end{proof}

\section{Upper bound}
\label{sec:upper}

We close with a very short treatment of the upper bound, which applies the
Planar Separator Theorem of Lipton and Tarjan \cite{LT}.

\begin{theorem} 
  \label{lem:planar_LT}
  [Lipton and Tarjan CITE]
  There is an absolute constant $c$ for which the following holds.  Every
  $n$-vertex planar graph can be partitioned into three sets $A$, $B$, and
  $C$ such that $|A| \leq 2n/3$, $|B| \leq 2n/3$, $|C| \leq c \sqrt n$, and
  there is no edge between $A$ and $B$.
\end{theorem}

\begin{proof}[Proof of Theorem~\ref{prop:upper}.]
  Initially, put $k \sqrt n$ cops at an arbitrary vertex, and call them
  \emph{free}\/ cops.  The constant $k$ will be determined at the end of the
  proof.  By continuously separating $G$ with these cops, we will show that
  the cops win the game.  Let $G_1 = G$, and for $i \geq 1$, construct
  $G_{i+1}$ with the following algorithm.

  Using Lemma~\ref{lem:planar_LT}, let $G_i = A_i \cup B_i \cup C_i$ such
  that $|A_i| \leq 2|G_i|/3$, $|B_i| \leq 2|G_i|/3$, $|C_i| \leq c \sqrt
  {|G_i|}$, and there is no edge in $G_i$ between $A_i$ and $B_i$.  Send
  free cops to each vertex of $C_i$ without regard to the robber's actions.
  Note that this is always possible because $G$ is strongly connected.
  This step costs at most $c \sqrt {|G_i|}$ free cops and these cops will
  not move afterward, thus becoming \emph{static}\/ cops.  If the robber was not
  caught during this process, the robber is now permanently trapped in
  either $A_i$ or $B_i$.  Set $G_{i+1}$ to be the graph induced by the
  corresponding set.  Thus, $|G_1| = n$, $|G_2| \leq 2n/3$, \ldots, $|G_i|
  \leq (\frac{2}{3})^{i - 1}n$.  The $i$-th step costs $c \sqrt {|G_i|}$
  free cops.  Thus, the number of cops needed to win is at most
  \[c \sum_{i = 0}^{\infty} \sqrt {\left(\frac{2}{3}\right)^i n} =
    c \sqrt n \sum_{i = 0}^{\infty} \left(\sqrt {\frac{2}{3}}\right) ^ i,
    \]
    which is at most $k \sqrt{n}$ for some $k$, because the sum is a
    convergent geometric series.
\end{proof}

\end{document}